\documentclass[11pt]{amsart}
\usepackage{fullpage}
\usepackage{amssymb,amsmath,amsfonts,amsthm,mathrsfs,verbatim,moreverb,xypic}
\usepackage{color}

\numberwithin{equation}{section}

\renewcommand{\AA}{\mathbb A}

\newcommand{\CC}{\mathbb C}

\newcommand{\FF}{\mathbb F}
\newcommand{\GG}{\mathbb G}

\newcommand{\PP}{\mathbb P}
\newcommand{\QQ}{\mathbb Q}
\newcommand{\RR}{\mathbb R}

\newcommand{\ZZ}{\mathbb Z}

\newcommand{\OO}{\mathcal O}

\newcommand{\aA}{\mathfrak a} 
\newcommand{\bB}{\mathfrak b} 
\newcommand{\p}{\mathfrak p}

\newcommand{\Pp}{\mathfrak P}

\newcommand{\Pic}{\operatorname{Pic}}

\newcommand{\ang}[1]{ \langle #1 \rangle  }

\def\ab{{\operatorname{ab}}}
\def\sep{{\operatorname{sep}}}
\def\perf{{\operatorname{perf}}}

\def\Gal{\operatorname{Gal}}
\def\ord{\operatorname{ord}}

\def\Aut{\operatorname{Aut}} 
\def\End{\operatorname{End}}
\def\Hom{\operatorname{Hom}}
\def\Frob{\operatorname{Frob}}

\newcommand{\defi}[1]{\textsf{#1}} 

\def\bbar#1{\setbox0=\hbox{$#1$}\dimen0=.2\ht0 \kern\dimen0 
\overline{\kern-\dimen0 #1}}
\newcommand{\Qbar}{{\overline{\mathbb Q}}} 
 
\newcommand{\Lbar}{\bbar{L}} 
\newcommand{\kbar}{\bbar{k}} 
\newcommand{\Fbar}{\bbar{F}} 
\newcommand{\FFbar}{\overline{\FF}} 

\def\twistedLaurent{ (\!(\tau^{-1})\!) }
\def\twistedseries{ [\![\tau^{-1}]\!] }

\newtheorem{thm}{Theorem}[section]
\newtheorem{lemma}[thm]{Lemma}
\newtheorem{cor}[thm]{Corollary}
\newtheorem{prop}[thm]{Proposition}

\theoremstyle{definition}
\newtheorem{definition}[thm]{Definition}

\theoremstyle{remark}
\newtheorem{remark}[thm]{Remark}

\newenvironment{romanenum}{\hfill \begin{enumerate} }{\end{enumerate}}

\definecolor{webcolor}{rgb}{0.8,0,0.2}
\definecolor{webbrown}{rgb}{.6,0,0}
\usepackage[
        colorlinks,
        linkcolor=webbrown,  filecolor=webcolor,  citecolor=webbrown, 
        backref,
        pdfauthor={David Zywina}, 
       pdftitle={},
]{hyperref}
\usepackage[alphabetic,backrefs,lite]{amsrefs} 

\begin{document}
\title[]{Explicit Class Field Theory for global function fields}

\author{David Zywina}
\address{Department of Mathematics and Statistics, Queen's University, Kingston, ON  K7L~3N6, Canada}
\email{zywina@mast.queensu.ca}
\urladdr{http://www.mast.queensu.ca/\~{}zywina}

\begin{abstract}
Let $F$ be a global function field and let $F^\ab$ be its maximal abelian extension.  Following an approach of D.~Hayes, we shall construct a continuous homomorphism $\rho\colon \Gal(F^\ab/F)\to C_F$, where $C_F$ is the idele class group of $F$.   Using class field theory, we shall show that our $\rho$ is an isomorphism of topological groups whose inverse is the Artin map of $F$.   As a consequence of the construction of $\rho$, we obtain an explicit description of $F^\ab$.      

Fix a place $\infty$ of $F$, and let $A$ be the subring of $F$ consisting of those elements which are regular away from $\infty$.  We construct $\rho$ by combining the Galois action on the torsion points of a suitable Drinfeld $A$-module with an associated $\infty$-adic representation studied by J.-K.~Yu.
\end{abstract}


\maketitle

\section{Introduction}

Let $F$ be a global field, that is, a finite field extension of either $\QQ$ or a function field $\FF_p(t)$.    Fix an algebraic closure $\Fbar$ of $F$, and let $F^\sep$ be the separable closure of $F$ in $\Fbar$.  \emph{Class field theory} gives a description of the maximal abelian extension $F^\ab$ of $F$ in $F^\sep$.   Let $\theta_F \colon  C_F \to \Gal(F^\ab/F)$ be the \defi{Artin map} where $C_F$ is the \defi{idele class group} of $F$; see \S\ref{SS:notation} for notation and \cite{tate-gcft} for background on class field theory.  The map $\theta_F$ is a continuous homomorphism.  By taking profinite completions, we obtain an isomorphism of topological groups:
\begin{equation*} \label{E:thetaF completion}
\widehat{\theta}_F \colon  \widehat C_F \xrightarrow{\sim} \Gal(F^\ab/F).
\end{equation*}
So $\theta_F$ gives a one-to-one correspondence between the finite abelian extensions $L$ of $F$ in $F^\sep$ and the finite index open subgroups of $C_F$.  For a finite abelian extension $L/F$, the corresponding open subgroup of $C_F$ is the kernel $U$ of the homomorphism 
\[
C_F\xrightarrow{\theta_F} \Gal(F^\ab/F)\twoheadrightarrow \Gal(L/F)
\]
where the second homomorphism is the restriction map $\sigma\mapsto \sigma|_L$ (the group $U$ is computable; it equals $N_{L/F}(C_L)$ where $N_{L/F}\colon C_L\to C_F$ is the norm map).   However, given an open subgroup $U$ with finite index in $C_F$, the Artin map ${\theta}_F$ does \emph{not} explicitly produce the corresponding extension field $L/F$ (in the sense, that it does not give a concrete set of generators for $L$ over $F$; though it does give enough information to uniquely characterize it).  \emph{Explicit class field theory} for $F$ entails giving constructions for all the abelian extensions of $F$.  We shall give a construction of $F^\ab$, for a global function field $F$, that at the same time produces the inverse of $\widehat\theta_F$ (without referring to $\widehat\theta_F$ or class field theory itself).\\

\subsection{Context}
D.~Hayes \cite{MR0330106} provided an explicit class field theory for rational function fields $F=k(t)$.   He built on and popularized the work of Carlitz from the 1930s \cite{MR1501937}, who had constructed a large class of abelian extensions using what is now called the \emph{Carlitz module} (this is the prototypical Drinfeld module; we will give the basic definitions concerning Drinfeld modules in \S\ref{S:background}).   Drinfeld and Hayes have both described explicit class field theory for an arbitrary global function field $F$, see \cite{MR0384707} and \cite{MR535766}.   Both proceed by first choosing a distinguished place $\infty$ of $F$, and their constructions give the maximal abelian extension $K_\infty$ of $F$ that splits completely at $\infty$.   Drinfeld defines a moduli space of rank 1 Drinfeld (elliptic) modules with level structure arising from the ``finite'' places of $F$ whose spectrum turns out to be $K_\infty$.  Hayes fixes a normalized Drinfeld module $\phi$ of rank 1 whose field of definition along with its torsion points can be used to construct $K_\infty$ (this approach is more favourable for explicit computations).

  One approach to computing the full field $F^\ab$ is to chose a second place $\infty'$ of $F$, since $F^\ab$ will then be the compositum of $K_\infty$ and $K_{\infty'}$.   We wish to give a version of explicit class field theory that does not require this second unnatural choice.   In Drinfeld's second paper on Drinfeld modules \cite{MR0439758}, he achieves exactly this my considering another moduli space of rank 1 Drinfeld modules but now with additional $\infty$-adic structure.    As remarked by Goss in \cite{MR1423131}*{\S7.5}, it would be very useful to have a modification of Drinfeld's construction that can be applied directly to $\phi$ to give the full abelian closure $F^\ab$.  We shall do exactly this!   J.-K.~Yu \cite{MR2018826} has studied the additional $\infty$-adic structure introduced by Drinfeld and has teased out the implicit Galois representation occurring there.  Yu's representation, which may also be defined for higher rank Drinfeld modules, can be viewed as an analogue of the Sato-Tate law, cf.~\cite{MR2018826} and \cite{Zywina-SatoTate}.

\subsection{Overview}

The goal of this paper is to given an explicit construction of the inverse $\widehat{\theta}_F$.     Moreover, we will construct an isomorphism of topological groups
\[
\rho\colon W_F^\ab \to C_F,
\]
where $W_F^\ab$ is the subgroup of $\Gal(F^\ab/F)$ that acts on the algebraic closure $\kbar$ of $k$ in $F^\sep$ as an integral power of Frobenius map $x\mapsto x^q$ (we endow $W_F^\ab$ with the weakest topology for which $\Gal(F^\ab/\kbar)$ is an open subgroup and has its usual profinite topology).   The inverse of our $\rho$ will be the homomorphism $\alpha\mapsto \theta_F(\alpha)^{-1}$ (Theorem~\ref{T:main}).     For an open subgroup $U$ of $C_F$ with finite index, define the homomorphism
\[
\rho_U\colon W_F^\ab \to C_F\twoheadrightarrow C_F/U;
\]
it factors through $\Gal(L_U/F)$ where $L_U$ is the field corresponding to $U$ via class field theory.  Everything about $\rho$ is computable, and in particular one can find generators for $L_U$.

In \S\ref{S:background}, we shall give the required background on Drinfeld modules; in particular, we focus on normalized Drinfeld modules of rank 1.  The representation $\rho$ will then be defined in \S\ref{SS:inverse}.     The rest of the introduction serves as further motivation and will not be needed later.  After a quick recap of explicit class field for $\QQ$, we shall describe the abelian extensions of $F=k(t)$ constructed by Carlitz which will lead to a characterization of $\rho_\infty$.  We will treat the  case $F=k(t)$ in more depth in \S\ref{S:k(t)}; in particular, we will recover Hayes' description of $k(t)^\ab$ as the compositum of three linearly disjoint fields.

\subsection{The rational numbers}

We briefly recall explicit class field theory for $\QQ$.  For each positive integer $n$, let $\mu_n$ be the subgroup of $n$-torsion points in $\Qbar^\times$; it is a free $\ZZ/n\ZZ$-module of rank $1$.  Let
\[
\chi_n \colon \Gal(\QQ^\ab/\QQ) \twoheadrightarrow (\ZZ/n\ZZ)^\times
\]
be the representation for which $\sigma(\zeta) = \zeta^{\chi_n(\sigma)}$ for all $\sigma \in \Gal(\QQ^\ab/\QQ)$ and $\zeta\in \mu_n$.  By taking the inverse image over all $n$, ordered by divisibility, we obtain a continuous and surjective representation 
\[
\chi \colon \Gal(\QQ^\ab/\QQ) \to \widehat{\ZZ}^\times.
\]
The \emph{Kronecker-Weber theorem} says that $\QQ^\ab$ is the cyclotomic extension of $\QQ$, and hence $\chi$ is an isomorphism of topological groups.  

The group $\widehat{\ZZ}^\times \times \RR^+$ can be viewed as a subgroup of $\AA_\QQ^\times$, where $\RR^+$ is the group of positive elements of $\RR^\times$.   The quotient map $\widehat{\ZZ}^\times \times \RR^+ \to C_\QQ$ is an isomorphism of topological groups, and by taking profinite completions we obtain an isomorphism $\widehat{\ZZ}^\times=\widehat{\ZZ}^\times\times \widehat{\RR^+} \xrightarrow{\sim} \widehat{C}_\QQ$.  Composing $\chi$ with this map, we have a isomorphism
\[
\rho\colon \Gal(\QQ^\ab/\QQ) \xrightarrow{\sim} \widehat{C}_\QQ.
\]
One can show that the inverse of the Artin map $\widehat{\theta}_\QQ\colon \widehat{C}_\QQ  \xrightarrow{\sim} \Gal(\QQ^\ab/\QQ)$ is simply the map $\sigma \mapsto \rho(\sigma)^{-1}$.

\subsection{The rational function field}
\label{SS:rational field}
Let us briefly consider the rational function field $F=k(t)$ where $k$ is a finite field with $q$ elements; it is the function field of the projective line $\PP^1_k$.  Let $\infty$ denote the place of $F$ for which $A=k[t]$ is the ring of regular function of $\PP^1_k$ that are regular away from $\infty$.

Let $\End_{k}(\GG_{a,F})$ be the ring of $k$-linear endomorphisms of the additive group scheme $\GG_{a}$ over $F$.   More concretely, $\End_k(\GG_{a,F})$ is the ring of polynomials $\sum_i c_i X^{q^i} \in F[X]$ with the usual addition and the multiplication operation being composition.  The \defi{Carlitz module} is the homomorphism
\[
\phi\colon A \to \End_k(\GG_{a,F}),\, a\mapsto \phi_a
\]
of $k$-algebras for which $\phi_t = t + X^q$.   Using the Carlitz module, we can give $F^\sep$ an interesting new $A$-module structure; i.e., for $a\in A$ and $\xi \in F^\sep$, we define $a\cdot \xi := \phi_a(\xi)$. 

For a monic polynomial $m\in A$, let $\phi[m]$ be the $m$-torsion subgroup of $F^\sep$, i.e., the set of $\xi \in F^\sep$ for which $m\cdot \xi=0$  (equivalently, the roots of the separable polynomial $\phi_m \in F[X]$).   The group $\phi[m]$ is free of rank $1$ as an $A/(m)$-module, and we have a continuous surjective homomorphism
\[
\chi_m \colon \Gal(F^\ab/F) \to (A/(m))^\times
\]
such that $\sigma(\xi) = \chi_m(\sigma)\cdot \xi$ for all $\sigma \in \Gal(F^\ab/F)$ and $\xi\in \phi[m]$.  By taking the inverse image over all monic $m\in A$, ordered by divisibility, we obtain a surjective continuous representation 
\[
\chi \colon \Gal(F^\ab/F) \to \widehat{A}^\times.
\]
However, unlike the cyclotomic case, the map $\chi$ is not an isomorphism.   The field $\bigcup_m F(\phi[m])$ is a geometric extension of $F$ that is tamely ramified at $\infty$; it does not contain extensions of $F$ that ramify at $\infty$ or the constant extensions.   

Following J.K.~Yu, and Drinfeld, we shall define a continuous homomorphism
\[
\rho_\infty \colon W_F^\ab \to F_\infty^\times
\]
where $F_\infty=k(\!(t^{-1})\!)$ is the completion of $F$ at the place $\infty$.   We will put off its definition (see \S\ref{S:k(t)} for further details in the $k(t)$ case), and simply note that $\rho_\infty$ can be characterized by the fact that it satisfies $\rho_\infty(\Frob_\p)=\p$ for each monic irreducible polynomial $\p$ of $A$.   The image of $\rho_\infty$ is contained in the open subgroup $F_\infty^+:= \ang{t}\cdot (1+t^{-1}k[\![t^{-1}]\!])$ of $F_\infty^\times$.  We can view $\widehat{A}^\times \times F_\infty^+$ as an open subgroup of the ideles $\AA_F^\times$.   We define the continuous homomorphism
\begin{equation*}
\rho\colon W_F^\ab \xrightarrow{\chi \times \rho_\infty } \widehat{A}^\times \times F_\infty^+ \to C_F
\end{equation*}
where the first map takes $\sigma\in W_F^\ab$ to $(\chi(\sigma),\rho_\infty(\sigma))$ and the second map is the compositum of the inclusion map to $\AA_F^\times$ with the quotient map $\AA_F^\times \to C_F$.

The main result of this paper, for $F=k(t)$, says that the above homomorphism $\rho \colon W_F^\ab \to C_F$ is an isomorphism of topological groups.  Moreover, the inverse of the Artin map $\theta_F\colon C_F\to W^\ab_F$ is the homomorphism $\sigma \mapsto \rho(\sigma)^{-1}$.   In particular, observe that $\rho$ does not depend on our initial choice of place $\infty$!  Taking profinite completions, we obtain an isomorphism $\widehat{\rho}\colon \Gal(F^\ab/F)\xrightarrow{\sim} \widehat{C}_F$.   \\

The constructions for a general global function field $F$ is more involved; it more closely resembles the theory of complex multiplication for elliptic curves than the cyclotomic theory.   We first choose a place $\infty$ of $F$.   In place of the Carlitz module, we will consider a suitable rank 1 Drinfeld module $\phi$.   We cannot always take $\phi$ to be defined over $F$, but we can choose a model defined over the maximal abelian extension $H_A$ of $F$ that is unramified at all places and splits completely at $\infty$ (we will actually work with a slightly larger field $H_A^+$).

\subsection{Notation}  \label{SS:notation}
 Throughout this paper, we consider a global function field $F$ with a fixed place $\infty$.   Let $A$ be the subring consisting of those functions that are regular away from $\infty$.   Denote by $k$ the field of constants of $F$ and let $q$ be the cardinality of $k$.  

For each place $\lambda$ of $F$, let $F_\lambda$ be the completion of $F$ at $\lambda$.   Let $\ord_\lambda \colon F_\lambda^\times \twoheadrightarrow \ZZ$ be the discrete valuation corresponding to $\lambda$ and let $\OO_\lambda\subseteq F_\lambda$ be its valuation ring.    The \defi{idele group} $\AA_F^\times$ of $F$ is the subgroup of $(\alpha_\lambda)\in \prod_\lambda F_\lambda^\times$, where the product is over all places of $F$, such that $\alpha_\lambda$ belongs to $\OO_\lambda^\times$ for all but finitely many places $\lambda$.   The group $\AA_F^\times$ is locally compact when endowed with the weakest topology for which $\prod_\lambda \OO_\lambda^\times$, with the profinite topology, is an open subgroup.  We embed $F^\times$ diagonally into $\AA_F^\times$; it is a discrete subgroup.   The \defi{idele class group} of $F$ is $C_F:=\AA_F^\times/F^\times$ which we endow with the quotient topology.

Let $\mathfrak{m}_\lambda$ be the maximal ideal of $\OO_\lambda$.   Define the finite field $\FF_\lambda=\OO_\lambda/\mathfrak{m}_\lambda$ whose cardinality we will denote by $N(\lambda)$.    The degree of the place $\infty$ is $d_\infty:= [\FF_\infty : k]$.

For a place $\p$ of $F$, we will denote by $\Frob_\p$ an arbitrary representative of the (arithmetic) Frobenius conjugacy class of $\p$ in $\Gal(F^\sep/F)$.\\

Let $L$ be an extension field of $k$.   We fix an algebraic closure $\Lbar$ of $L$ and let $L^\sep$ be the separable closure of $L$ in $\Lbar$.  We shall take $\kbar$ to be the the algebraic closure of $k$ in $L^\sep$.  Let $L^\perf$ be the perfect closure of $L$ in $\Lbar$.   We shall denote the absolute Galois group of $L$ by $\Gal_L:=\Gal(L^\sep/L)$.   The \defi{Weil group} $W_L$ of $L$ is the subgroup of $\Gal_L$ consisting of those automorphisms $\sigma$ for which there exists an integer $\deg(\sigma)$ that satisfy $\sigma(x)=x^{q^{\deg(\sigma)}}$ for all $x\in \kbar$.  The map $\deg \colon W_L \to \ZZ$ is a group homomorphism with kernel $\Gal(L^\sep/L\kbar)$.     We endow $W_L$ with the weakest topology for which $\Gal(L^\sep/L\kbar)$ is an open subgroup with its usual profinite topology.   Let $W_L^\ab$ be the image of $W_L$ under the restriction map $\Gal_L\to \Gal(L^\ab/L)$ where $L^\ab$ is the maximal abelian extension of $L$ in $L^\sep$.\\

Let $L$ be an extension field of $k$.  We define $L[\tau]$ be the ring of polynomials in $\tau$ with coefficients in $L$ that obey the commutation rule $\tau \cdot a = a^q \tau$ for $a\in L$.  In particular, note that $L[\tau]$ will be non-commutative if $L\neq k$.  We can identify $L[\tau]$ with the $k$-algebra $\End_{k}(\GG_{a,L})$ consisting of the $k$-linear endomorphism of the additive group scheme $\GG_{a,L}$; identify $\tau$ with the endomorphism $X\mapsto X^q$.

Suppose that $L$ is perfect.   Let $L\twistedLaurent$ be the skew-field consisting of Laurent series in $\tau^{-1}$; it contains $L[\tau]$ as a subring  (we need $L$ to be perfect so that $\tau^{-1} \cdot a = a^{1/q} \tau$ is always valid).    Define the valuation $\ord_{\tau^{-1}}\colon L\twistedLaurent\to \ZZ\cup\{+\infty\}$ by $\ord_{\tau^{-1}}(\sum_i a_i \tau^{-i}) = \inf\{ i : a_i \neq 0 \}$ and $\ord_{\tau^{-1}}(0)=+\infty$.   The valuation ring of $\ord_{\tau^{-1}}$ is $L\twistedseries$, i.e., the ring of formal power series in $\tau^{-1}$.  Again note that $L\twistedLaurent$ and $L\twistedseries$ are non-commutative if $L\neq k$.

For a topological group $G$, we will denote by $\widehat{G}$ the profinite completion of $G$.   We will always consider profinite groups, for example Galois groups, with their profinite topology.    

\subsection*{Acknowledgements} Thanks to David Goss and Bjorn Poonen.

\section{Background} \label{S:background}
For an in-depth introduction to Drinfeld modules, see \cite{MR0384707,MR902591,MR1423131}.  The introduction \cite{MR1196509} of Hayes and Chapter VII of \cite{MR1423131} are particularly relevant to the material of this section. 

\subsection{Drinfeld modules}
Let $L$ be a field extension of $k$.  A \defi{Drinfeld module} over $L$ is a homomorphism $\phi\colon A \to L[\tau],\, x\mapsto \phi_x$ of $k$-algebras whose image contains a non-constant polynomial.  Let $\partial\colon L[\tau] \to L$ be the ring homomorphism that takes a twisted polynomial to its constant term.   We say that $\phi$ has \defi{generic characteristic} if the homomorphism $\partial \circ \phi \colon A \to L$ is injective; using this map, we then obtain an embedding $F\hookrightarrow L$ that we will view as an inclusion. 

The ring $L[\tau]$ is contained in the skew field $L^\perf\twistedLaurent$.  The map $\phi$ is injective, so it extends uniquely to a homomorphism $\phi\colon F\hookrightarrow L^\perf\twistedLaurent$.   The function $v\colon F\to \ZZ\cup\{+\infty\}$ defined by $v(x)=\ord_{\tau^{-1}}(\phi_x)$ is a discrete valuation that satisfies $v(x)\leq 0$ for all non-zero $x\in A$; the valuation $v$ is non-trivial since the image of $\phi$ contain a non-constant element of $L[\tau]$.  Therefore, $v$ is equivalent to $\ord_\infty$ and hence there exists a positive $r\in\QQ$ that satisfies
\begin{equation} \label{E:rank defn}
\ord_{\tau^{-1}}(\phi_x) = r d_\infty \ord_\infty(x) 
\end{equation}
for all $x\in F^\times$.   The number $r$ is called the \defi{rank} of $\phi$ and it is always an integer.  

Since $L^\perf\twistedLaurent$ is complete with respect to $\ord_{\tau^{-1}}$, the map $\phi$ extends uniquely to a homomorphism
\[
\phi\colon F_\infty\hookrightarrow L^\perf\twistedLaurent
\]
that satisfies (\ref{E:rank defn}) for all $x\in F_\infty^\times$.  This extension of $\phi$ was first constructed in \cite{MR0439758} and will be the key to the $\infty$-adic part of the construction of the inverse of the Artin map in \S\ref{SS:infty-rep}.   It will also lead to a more straightforward definition of the ``leading coefficient'' map $\mu_\phi$ of \cite{MR1196509}*{\S6}, see \S\ref{SS:normalization}.

Restricting our extended map $\phi$ to $\FF_\infty$ gives a homomorphism $\FF_\infty \to L^\perf \twistedseries$.   After composing $\phi|_{\FF_\infty}$ with the homomorphism $L^\perf\twistedseries\to L^\perf$ which takes a power series in $\tau^{-1}$ to its constant term, we obtain a homomorphism $\FF_\infty \hookrightarrow L^\perf$ of $k$-algebras whose image must lie in $L$.  In particular, the Drinfeld module $\phi$ gives $L$ the structure of an $\FF_\infty$-algebra.

Given two Drinfeld modules $\phi,\phi'\colon A \to L[\tau]$, an \defi{isogeny from $\phi$ to $\phi'$ over $L$} is a non-zero $f\in L[\tau]$ for which $f \phi_a = \phi'_a f$ for all $a\in A$.  

\subsection{Normalized Drinfeld modules} \label{SS:normalization}

Fix a Drinfeld module $\phi\colon A \to L[\tau]$ of rank $r$ and also denote by $\phi$ its extension $F_\infty \to L^\perf\twistedLaurent$.   For each $x\in F_\infty^\times$, we define $\mu_\phi(x)\in (L^\perf)^\times$ to be the first non-zero coefficient of the Laurent series $\phi_x \in L^\perf\twistedLaurent$.  By (\ref{E:rank defn}), the first term of $\phi_x$ is $\mu_\phi(x)\tau^{-rd_\infty\ord_\infty(x)}$.  For a non-zero $x\in A$, one can also define $\mu_\phi(x)$ as the leading coefficient of $\phi_x$ as a polynomial in $\tau$.    

For $x,y\in F_\infty^\times$, the value $\mu_\phi(xy)$ is equal to the coefficient of $\mu_ \phi(x) \tau^{-rd_\infty\ord_{\infty}(x)} \cdot  \mu_ \phi(y) \tau^{-rd_\infty\ord_{\infty}(y)}$, and hence  
\begin{equation} \label{E:sign relation}
\mu_ \phi(xy)=\mu_ \phi(x) \mu_ \phi(y)^{1/q^{r d_\infty\ord_\infty(x)}}.
\end{equation}
With respect to our embedding $\FF_\infty\hookrightarrow L$ arising from $\phi$, we have $\mu_\phi(x)=x$ for all $x\in \FF_\infty^\times$. 

We say that $\phi$ is \defi{normalized} if $\mu_\phi(x)$ belongs to $\FF_\infty^\times$ for all $x\in F_\infty^\times$ (equivalently, for all non-zero $x\in A$).  If $\phi$ is normalized, then by (\ref{E:sign relation}) the map $\mu_\phi\colon F_\infty^\times \to \FF_\infty^\times$ is a group homomorphism that equals the identity map when restricted to $\FF_\infty^\times$; this is an example of a sign function.  

\begin{definition}
A \defi{sign function} for $F_\infty$ is a group homomorphism $\varepsilon\colon F_\infty^\times \to \FF_\infty^\times$ that is the identity map when restricted to $\FF_\infty^\times$.  We say that $\phi$ is \defi{$\varepsilon$-normalized} if it is normalized and $\mu_\phi\colon F_\infty^\times \to \FF_\infty^\times$ is equal to $\varepsilon$ composed with some $k$-automorphism of $\FF_\infty$.
\end{definition}

A sign function $\varepsilon$ is trivial on $1+\mathfrak{m}_\infty$, so it determined by the value $\varepsilon(\pi)$ for a fixed uniformizer $\pi$ of $F_\infty$.

\begin{lemma} \label{L:normalization works} \cite{MR1196509}*{\S12}
Let $\varepsilon$ be a sign function for $F_\infty$ and let $\phi'\colon A\to L[\tau]$ be a Drinfeld module.   Then $\phi'$ is isomorphic over $\Lbar$ to an $\varepsilon$-normalized Drinfeld module $\phi\colon A\to \Lbar[\tau]$.
\end{lemma}

\subsection{The action of an ideal on a Drinfeld module}
Fix a Drinfeld module $\phi\colon A \to L[\tau]$ and take a non-zero ideal $\aA$ of $A$.  Let $I_{\aA,\phi}$ be the \emph{left} ideal in $L[\tau]$ generated by the set $\{\phi_a: a \in \aA\}$.   All left ideals of $L[\tau]$ are principal, so $I_{\aA,\phi}=L[\tau]\cdot \phi_{\aA}$ for a unique monic polynomial $\phi_{\aA}\in L[\tau]$.  Using that $\aA$ is an ideal, we find that $I_{\aA,\phi} \phi_x \subseteq I_{\aA,\phi}$ for all $x\in A$.   Thus for each $x\in A$, there is a unique polynomial $(\aA*\phi)_x$ in $L[\tau]$ that satisfies 
 \begin{equation} \label{E:a isogeny}
 \phi_\aA\cdot \phi_x = (\aA*\phi)_x \cdot \phi_\aA.
 \end{equation}    
The map
 \[
 \aA* \phi \colon A \to L[\tau],\,\, x\mapsto (\aA* \phi)_x
 \]
 is also a Drinfeld module, and hence (\ref{E:a isogeny}) shows that $\phi_\aA$ is an isogeny from $\phi$ to $\aA* \phi$.  
 
\begin{lemma} \label{L:ideal action} \cite{MR1196509}*{\S4}
\begin{romanenum}
\item \label{I:ideal action, multiplication} Let $\aA$ and $\bB$ be non-zero ideals of $A$.  Then $\phi_{\aA\bB}=(\bB*\phi)_\aA\cdot \phi_\bB$ and $\aA*(\bB*\phi)=(\aA\bB)*\phi$.
\item \label{I:ideal action, principal}  Let $\aA=wA$ be a non-zero principal ideal of $A$.  Then $\phi_\aA=\mu_\phi(w)^{-1} \cdot \phi_w$ and  $(\aA*\phi)_x = \mu_\phi(w)^{-1}\cdot \phi_x\cdot  \mu_\phi(w)$ for all $x\in A$.
\end{romanenum}
\end{lemma}

\begin{lemma} \label{L:ideal-Galois actions}
Let $\sigma\colon L\hookrightarrow L'$ be an embedding of fields.  Let $\sigma(\phi)\colon A \to L'[\tau]$ be the Drinfeld module for which $\sigma(\phi)_x = \sigma(\phi_x)$, where $\sigma$ acts on the coefficients of $L[\tau]$.  For each non-zero ideal $\aA$ of $A$, we have $\sigma(\aA*\phi)=\aA*\sigma(\phi)$ and $\sigma(\phi_\aA)=\sigma(\phi)_\aA$.
\end{lemma}
\begin{proof}
The left ideal of $L'[\tau]$ generated by $\sigma(I_{\aA,\phi})$ is $I_{\aA,\sigma(\phi)}$, and hence $\sigma(\phi_\aA)=\sigma(\phi)_\aA$.  Applying $\sigma$ to (\ref{E:a isogeny}), we have
\[
\sigma(\phi)_\aA\sigma(\phi)_x = \sigma(\phi_\aA\cdot \phi_x) = \sigma((\aA*\phi)_x \cdot \phi_\aA) = \sigma(\aA*\phi)_x \sigma(\phi)_\aA,
\]
for all $x\in A$.  This shows that $\sigma(\aA*\phi)=\aA*\sigma(\phi)$.
\end{proof}

\subsection{Hayes modules} \label{SS:Hayes}
Fix a sign function $\varepsilon$ for $F_\infty$.  Let $\CC_\infty$ be the completion of an algebraic closure $\Fbar_\infty$ of $F_\infty$ with respect to the $\infty$-adic norm; it is both complete and algebraically closed.

\begin{definition} 
  A \defi{Hayes module} for $\varepsilon$ is a Drinfeld module $\phi\colon A \to \CC_\infty[\tau]$ of rank 1 that is $\varepsilon$-normalized and for which $\partial \circ \phi\colon A \to \CC_\infty$ is the inclusion map.  Denote by $X_\varepsilon$ the set of Hayes modules for $\varepsilon$.
\end{definition}

We know that Hayes modules exist because Drinfeld $A$-modules over $\CC_\infty$ of rank 1 can be constructed analytically \cite{MR0384707}*{\S3} and then we can apply Lemma~\ref{L:normalization works}.\footnote{In our construction of the inverse of the Artin map, this is the only part that we have not made explicit.  It is analogous to how one analytically constructs an elliptic curve with complex multiplication by the ring of integers $\OO_K$ of a quadratic imaginary field $K$ [The quotient $\CC/\OO_K$ gives such an elliptic curve over $\CC$.  We can then compute the $j$-invariant to high enough precision to identify what algebraic integer it is (it belongs to the Hilbert class field of $K$, so we know its degree over $\QQ$).].  The current version of \texttt{Magma} \cite{Magma} has a function \texttt{AnalyticDrinfeldModule} that can compute rank 1 Drinfeld modules that are defined over the maximal abelian extension $H_A$ of $F$ that is unramified at all places and splits completely at $\infty$.}

Take any Hayes module $\phi\in X_\varepsilon$.   Using that $\phi_\aA \in \CC_\infty[\tau]$ is monic along with (\ref{E:a isogeny}), we see that the Drinfeld module $\aA*\phi$ also belongs to $X_\varepsilon$.  By Lemma~\ref{L:ideal action}(\ref{I:ideal action, multiplication}), we find that the group $\mathcal{I}$ of fractional ideals of $A$ acts on $X_\varepsilon$.   Let $\mathcal{P}^+$ be the subgroup of principal fractional ideals generated by those $x\in F^\times$ that satisfy $\varepsilon(x)=1$.   The group $\mathcal{P}^+$ acts trivially on $X_\varepsilon$ by Lemma~\ref{L:ideal action}(\ref{I:ideal action, principal}), and hence induces an action of the finite group $\Pic^+(A):=\mathcal{I}/\mathcal{P}^+$ on $X_\varepsilon$.

\begin{prop} \label{P:homogenous space} \cite{MR1196509}*{\S13}
The set $X_\varepsilon$ is a principal homogeneous space for $\Pic^+(A)$ under the $*$ action.
\end{prop}

We now consider the arithmetic of the set $X_\varepsilon$.   Take any $\phi\in X_\varepsilon$ and choose a non-constant $y\in A$.  Let $H_A^+$ be the subfield of $\CC_\infty$ generated by $F$ and the coefficients of $\phi_y$ as a polynomial in $\tau$.    We call $H_A^+$ the \defi{normalizing field} for the triple $(F,\infty,\varepsilon)$.

\begin{lemma} \cite{MR1196509}*{\S14}
The extension $H_A^+/F$ is finite and normal, and depends only on the triple $(F,\infty,\varepsilon)$.   
\end{lemma}

So for every $\phi\in X_\varepsilon$, we have $\phi(A)\subseteq H^+_A[\tau]$.   From now on, we shall view $\phi$ as a Drinfeld module $\phi\colon A\to H^+_A[\tau]$.    Moreover, we actually have $\phi(A)\subseteq B[\tau]$ where $B$ is the integral closure of $A$ in $H_A^+$.  Since $\phi$ is normalized, we find that $\phi$ has \defi{good reduction} at every place of $H_A^+$ not lying over $\infty$ (for each non-zero prime ideal $\Pp$ of $B$, we can compose $\phi$ with a reduction modulo $\Pp$ map to obtain a Drinfeld module of rank 1 over $B/\Pp$).

There is thus a natural action of the Galois group $\Gal(H_A^+/F)$ on $X_\varepsilon$.   With a fixed $\phi \in X_\varepsilon$,  Proposition~\ref{P:homogenous space} implies that there is a unique function $\psi\colon \Gal(H_A^+/F) \hookrightarrow \Pic^+(A)$ such that $\sigma(\phi)=\psi(\sigma)*\phi$.   Lemma~\ref{L:ideal action} and Lemma~\ref{L:ideal-Galois actions} imply that $\psi$ is a group homomorphism that does not depend on the initial choice of $\phi$.  A consequence of following important proposition is that $\psi$ is surjective, and hence an isomorphism.

\begin{prop} \label{P:Frob for HA+}
The extension $H_A^+/F$ is unramified away from $\infty$.  For each non-zero prime ideal $\p$ of $A$,  the class  $\psi(\Frob_\p)$ of $\Pic^+(A)$ is the one containing $\p$.
\end{prop}

\section{Construction of the inverse of the Artin map} \label{S:construction}

Fix a place $\infty$ of $F$.  Throughout this section, we also fix a sign function $\varepsilon$ for $F_\infty$ and a Hayes module $\phi \in X_\varepsilon$ (as described in the previous section).    Let $F_\infty^+$ be the open subgroup of $F_\infty^\times$ consisting of those $x\in F_\infty^\times$ for which $\varepsilon(x)=1$.   So as not to clutter the construction, all the lemmas of \S\ref{S:construction} will be proved in \S\ref{S:lemma proofs}.

\subsection{$\lambda$-adic representations}
Fix a place $\lambda\neq \infty$ of $F$; we shall also denote by $\lambda$ the corresponding maximal ideal of $A$.   Take any automorphism $\sigma \in \Gal_F$.   Since the map $\psi$ of \S\ref{SS:Hayes} is surjective, we can choose a non-zero ideal $\aA$ of $A$ for which $\sigma(\phi)=\aA * \phi$.

For each positive integer $e$, let $\phi[\lambda^e]$ be the set of $b\in \Fbar$ that satisfy $\phi_x(b)=0$ for all $x\in \lambda^e$; equivalently, $\phi[\lambda^e]$ is the set of $b\in \Fbar$ such that $\phi_{\lambda^e}(b)=0$ (recall that we can identify each element of $L[\tau]$ with a unique polynomial $\sum_{i\geq 0} c_i X^{q^i} \in L[X]$).   We have $\phi[\lambda^e]\subseteq F^\sep$ since the polynomials $\phi_x$ are separable for all $x\in \lambda^e$.  Using the $A$-module structure coming from $\phi$, we find that $\phi[\lambda^e]$ is an $A/\lambda^e$-module of rank 1.   The \defi{$\lambda$-adic Tate module} of $\phi$ is defined to be
\[
T_\lambda(\phi)=\Hom_A(F_\lambda/\OO_\lambda,\, \phi[\lambda^\infty])
\]
where $\phi[\lambda^\infty]= \cup_{e\geq 1} \phi[\lambda^e]$.   The Tate module $T_\lambda(\phi)$ is a free $\OO_\lambda$-module of rank $1$, and hence $V_\lambda(\phi):=F_\lambda\otimes_{\OO_\lambda} T_\lambda(\phi)$ is a one-dimensional vector space over $F_\lambda$.   

For each $e$, the map $\phi[\lambda^e]\to \sigma(\phi)[\lambda^e],\, \xi\mapsto \sigma(\xi)$ is an isomorphism of $A/\lambda^e$-modules.  Combining over all $e$, we obtain an isomorphism $V_\lambda(\sigma)\colon V_\lambda(\phi)\to V_\lambda(\sigma(\phi))$ of $F_\lambda$-vector spaces.   

The isogeny $\phi_\aA$ from $\phi$ to $\aA*\phi$ induces a homomorphism $\phi[\lambda^e]\to (\aA*\phi)[\lambda^e]$ of $A/\lambda^e$-module for each $e$.  Combining together, we obtain an isomorphism $V_\lambda(\phi_\aA)\colon V_\lambda(\phi)\to V_\lambda(\aA*\phi)$ of $F_\lambda$-vector spaces.    

Using our assumption $\sigma(\phi)=\aA*\phi$, the map $V_\lambda(\phi_\aA)^{-1}\circ V_\lambda(\sigma)$ belongs to $\Aut_{F_\lambda}(V_\lambda(\phi))=F_\lambda^\times$; we denote this element of $F_\lambda^\times$ by $\rho_\lambda^{\aA}(\sigma)$. 

\begin{lemma}  \label{L:lambda basics}
\begin{romanenum}
\item \label{I:lambda basics-mult}
Take $\sigma,\gamma \in \Gal_F$ and fix ideals $\aA$ and $\bB$ of $A$ such that $\sigma(\phi)=\aA*\phi$ and $\gamma(\phi)=\bB*\phi$.   Then $(\sigma\gamma)(\phi)=(\aA\bB)*\phi$ and $\rho^{\aA\bB}_\lambda(\sigma\gamma)=\rho^{\aA}_\lambda(\sigma) \rho^{\bB}_\lambda(\gamma)$.
\item \label{I:lambda basics-defined}
Take $\sigma\in \Gal_F$ and fix ideals $\aA$ and $\bB$ of $A$ such that $\sigma(\phi)=\aA*\phi=\bB*\phi$. Then $\rho_\lambda^{\aA}(\sigma) \rho_\lambda^{\bB}(\sigma)^{-1}$ is the unique generator $w\in F^\times$ of the fractional ideal $\bB\aA^{-1}$ that satisfies $\varepsilon(w)=1$.
\item \label{I:valuation of automorphism}
Take $\sigma\in \Gal_F$ and fix an ideal $\aA$ such that $\sigma(\phi)=\aA*\phi$.    Identifying $\lambda$ with a non-zero prime ideal of $A$, let $f\geq 0$ be the largest power of $\lambda$ dividing $\aA$.  Then $\ord_\lambda(\rho^{\aA}_\lambda(\sigma))=-f$.
\end{romanenum}
\end{lemma}

By Lemma~\ref{L:lambda basics}(\ref{I:lambda basics-mult}), the map
\[
\rho_\lambda \colon \Gal_{H_A^+}\to \OO_\lambda^\times,\quad \sigma\mapsto \rho^{A}_\lambda(\sigma)
\]
is a homomorphism.   It is a continuous Galois representation and is unramified at all places not lying over $\lambda$ or $\infty$ (recall that $\phi$ has good reduction at all places not dividing $\infty$).   

\subsection{$\infty$-adic representation}   \label{SS:infty-rep}
By \S\ref{SS:normalization}, our Drinfeld module $\phi\colon A \to H_A^+[\tau]$ extends uniquely to a homomorphism $\phi\colon F_\infty \hookrightarrow (H_A^+)^\perf\twistedLaurent$ that satisfies (\ref{E:rank defn}) with $r=1$ for all $x\in F_\infty^\times$.    Recall that we defined a homomorphism $\deg \colon W_F \twoheadrightarrow \ZZ$ for which $\sigma(x)$ equals $x^{q^{\deg(\sigma)}}=\tau^{\deg(\sigma)} x \tau^{-\deg(\sigma)}$ for all $x\in \kbar$.   Given a series $u\in \Fbar\twistedLaurent$ with coefficients in $F^\sep$ and an automorphism $\sigma\in \Gal_F$, we define $\sigma(u)$ to be the series obtained by applying $\sigma$ to the coefficients of $u$.

\begin{lemma} \label{L:u facts}
\begin{romanenum}
\item \label{I:u exists}
There exists a series $u \in \Fbar\twistedseries^\times$ such that $u^{-1}\phi(F_\infty) u \subseteq \kbar\twistedLaurent$.  Any such $u$ has coefficients in $F^\sep$.
\item \label{I:u independence}
Fix any $u$ as in (\ref{I:u exists}). Take any $\sigma\in W_F$ and fix a non-zero ideal $\aA$ of $A$ for which $\sigma(\phi)=\aA*\phi$.  Then
\[
 \phi_{\aA}^{-1} \sigma(u) \tau^{\deg(\sigma)} u^{-1} \in \Fbar\twistedLaurent
\]
belongs to $\phi(F_\infty^+)$ and is independent of the choice of $u$.
\end{romanenum}
\end{lemma}

Take $\sigma\in W_F$ and fix a non-zero ideal $\aA$ of $A$ such that $\sigma(\phi)=\aA*\phi$.  Choose a series $u \in \Fbar\twistedseries^\times$ as in Lemma \ref{L:u facts}(\ref{I:u exists}).
Using that $\phi\colon F_\infty \to (H_A^+)^\perf\twistedLaurent$ is injective and Lemma~\ref{L:u facts}(\ref{I:u independence}), we define $\rho_\infty^\aA(\sigma)$ to be the unique element of $F_\infty^+$ for which
\[
\phi\big(\rho_\infty^\aA(\sigma)\big)=\phi_{\aA}^{-1} \sigma(u) \tau^{\deg(\sigma)} u^{-1}.
\]
We now state some results about $\rho_\infty^\aA(\sigma)$; they are analogous to those concerning $\rho_\lambda^\aA(\sigma)$ in the previous section.

\begin{lemma} \label{L:infty basics}
\begin{romanenum}
\item \label{I:infty basics-mult}
Take $\sigma,\gamma \in W_F$ and fix ideals $\aA$ and $\bB$ of $A$ such that $\sigma(\phi)=\aA*\phi$ and $\gamma(\phi)=\bB*\phi$.   Then $(\sigma\gamma)(\phi)=(\aA\bB)*\phi$, and we have $\rho^{\aA\bB}_\infty(\sigma\gamma)=\rho^{\aA}_\infty(\sigma) \rho^{\bB}_\infty(\gamma)$.
\item \label{I:infty basics-defined}
Take $\sigma\in W_F$ and fix ideals $\aA$ and $\bB$ of $A$ such that $\sigma(\phi)=\aA*\phi=\bB*\phi$.  Then $\rho_\infty^{\aA}(\sigma) \rho_\infty^{\bB}(\sigma)^{-1}$ is the unique generator $w\in F^\times$ of the fractional ideal $\bB\aA^{-1}$ that satisfies $\varepsilon(w)=1$.
\end{romanenum}
\end{lemma}

\begin{lemma} \label{L:rho-infty is continuous}
The map
\[
\rho_\infty \colon W_{H_A^+}\to F_\infty^+,\quad \sigma\mapsto \rho^{A}_\infty(\sigma)
\]
is a continuous homomorphism that is unramified at all places of $H_A^+$ not lying over $\infty$.
\end{lemma}

\subsection{The inverse of the Artin map} \label{SS:inverse}

For each $\sigma\in W_F$, fix a non-zero ideal $\aA$ of $A$ such that $\sigma(\phi)=\aA*\phi$.  By Lemma~\ref{L:lambda basics}(\ref{I:valuation of automorphism}), $(\rho^\aA_\lambda(\sigma))_\lambda$ is an idele of $F$.   We define $\rho(\sigma)$ to be the element of the idele class group $C_F$ that is represented by $(\rho^\aA_\lambda(\sigma))_\lambda$.  By Lemma~\ref{L:lambda basics}(\ref{I:lambda basics-defined}) and Lemma~\ref{L:infty basics}(\ref{I:infty basics-defined}), we find that $\rho(\sigma)$ is independent of the choice of $\aA$.     Lemma~\ref{L:lambda basics}(\ref{I:lambda basics-mult}) and Lemma~\ref{L:infty basics}(\ref{I:infty basics-mult}) imply that the map
\[
\rho\colon W_F \to C_F
\]
is a group homomorphism.    The restriction of $\rho$ to the finite index open subgroup $W_{H_A^+}$ agrees with 
\[
W_{H^+_A} \xrightarrow{\prod_\lambda \rho_\lambda} F^+_\infty \times \prod_{\lambda\neq \infty} \OO_\lambda^\times \hookrightarrow C_F
\]
where the second homomorphism is obtained by composing the natural inclusion into $\AA_F^\times$ with the quotient map $\AA_F^\times\to C_F$.  Since the representations $\rho_\lambda$ are continuous, we deduce that $\prod_\lambda\rho_\lambda$, and hence $\rho$, is continuous.  So we may view $\rho$ as a continuous homomorphism $W_F^\ab \to C_F$.   By taking profinite completions, we obtain a continuous homomorphism
\[
\widehat{\rho} \colon \Gal(F^\ab/F) \to \widehat C_F.
\]
Recall that the Artin map ${\theta}_F \colon  C_F \to \Gal(F^\ab/F)$ of class field theory gives an isomorphism
\[
\widehat{\theta}_F \colon \widehat C_F \xrightarrow{\sim} \Gal(F^\ab/F)
\]
of topological groups.  Our main result is then the following:

\begin{thm} \label{T:main}
The map $\widehat\rho\colon \Gal(F^\ab/F) \to \widehat C_F$ is an isomorphism of topological groups.  The inverse of the isomorphism $\Gal(F^\ab/F) \to \widehat C_F,\, \sigma \mapsto \widehat\rho(\sigma)^{-1}$ is the Artin map $\widehat{\theta}_F$.
\end{thm}

Before proving the theorem, we mention the following arithmetic input.

\begin{lemma} \label{L:Frobenius}
Fix a place $\lambda$ of $F$.  Let $\p\neq \lambda, \infty$ be a place of $F$ which we identify with the corresponding non-zero prime ideal of $A$.  Then $\rho^{\p}_\lambda(\Frob_\p)=1$. 
\end{lemma}

\begin{proof}[Proof of Theorem~\ref{T:main}]
Take an open subgroup $U$ of $C_F$ with finite index.   Let $L_U$ be the fixed field in $F^\sep$ of the kernel of the homomorphism $W_F^\ab \xrightarrow{\rho} C_F\twoheadrightarrow C_F/U$; this gives an injective group homomorphism $\rho_U\colon \Gal(L_U/F) \hookrightarrow C_F/U$.   Let $S_U$ be the set of places $\p$ of $F$ for which $\p=\infty$ or for which there exists an idele $\alpha \in \AA_F^\times$ whose class in $C_F$ does not lie in $U$ and satisfying $\alpha_\lambda=1$ for $\lambda\neq \p$ and $\alpha_\p \in \OO_\p^\times$.  The set $S_U$ is finite since $U$ is open in $C_F$.

Take any place $\p\notin S_U$.  Choose a uniformizer $\pi_\p$ of $F_\p$ and let $\alpha(\p)$ be the idele of $F$ that is $\pi_\p$ at the place $\p$ and 1 at all other places.    Define the idele $\beta:=(\rho^\p_\lambda(\Frob_\p))_\lambda \cdot \alpha(\p) \in \AA_F^\times$.   Lemma~\ref{L:Frobenius} says that $\beta_\lambda=1$ for all $\lambda\neq \p$ while Lemma~\ref{L:lambda basics}(\ref{I:valuation of automorphism}) tells us that $\ord_\p(\beta_\p)=0$.   By our choice of $S_U$, the image of $\beta$ in $C_F$ must lie in $U$.   Therefore, $\rho_U(\Frob_\p)$ is the coset of $C_F/U$ represented by $\alpha(\p)^{-1}$.    In particular, note that $L_U/F$ is unramified at all $\p \notin S_U$.  The group $C_F/U$ is generated by the elements $\alpha(\p)$ with $\p\notin S_U$, and hence $\rho_U$ is surjective.   Therefore $\rho_U\colon \Gal(L_U/F) \xrightarrow{} C_F/U$ is an isomorphism of groups.

Define the isomorphism $\theta_U\colon C_F/U\xrightarrow{\sim} \Gal(L_U/F), \, \alpha \mapsto (\rho_U^{-1}(\alpha))^{-1}$.  For each $\p \notin S_U$, it takes the coset containing $\alpha(\p)$ to the Frobenius automorphism corresponding to $\p$.   Composing  the quotient map $C_F\to C_F/U$ with $\theta_U$, we find that the resulting homomorphism $C_F\to \Gal(L_U/F)$ equals the map $\alpha\mapsto \theta_F(\alpha)|_{L_U}$ where $\theta_F\colon C_F\to \Gal(F^\ab/F)$ is the Artin map of $F$.

Recall that class field theory gives a one-to-one correspondence between the finite abelian extensions $L$ of $F$ and the open subgroups $U$ with finite index in $C_F$.   Let $L\subseteq F^\sep$ be an arbitrary finite abelian extension of $F$.  Class field theory says that $L$ corresponds to the kernel $U$ of the map $C_F\to \Gal(L/F),\,\alpha\mapsto \theta_F(\alpha)|_L$.    By comparing with the computation above, we deduce that $L=L_U$.     Since $L$ was an arbitrary finite abelian extension, we deduce that $F^\ab = \bigcup_U L_U$.

Taking the inverse limit of the isomorphisms $\rho_U\colon \Gal(L_U/F)\xrightarrow{\sim} C_F/U$ as $U$ varies, we find that the corresponding homomorphism $\rho\colon \Gal(F^\ab/F) \to \widehat{C}_F$ is an isomorphism (the injectivity is precisely the statement that $F^\ab = \bigcup_U L_U$).  The inverse of the isomorphism $\Gal(F^\ab/F) \to \widehat C_F,\, \sigma \mapsto \widehat\rho(\sigma)^{-1}$ is obtained by combining the homomorphisms $\theta_U\colon C_F/U \xrightarrow{\sim} \Gal(L_U/F)$; from the calculation above, this equals $\widehat{\theta}_F$.
\end{proof}

\begin{cor} \label{C:main}
The homomorphism $\rho\colon W_F^\ab\to C_F$ is an isomorphism of topological groups.   The inverse of the isomorphism $W_F^\ab \to C_F,\, \sigma \mapsto \rho(\sigma)^{-1}$ is the Artin map ${\theta}_F$.
\end{cor}
\begin{proof}
This follows directly from the theorem.  Observe that the natural maps $W_F^\ab \to \widehat{W_F^\ab}=\Gal(F^\ab/F)$ and $C_F\to \widehat{C}_F$ from the group to their profinite completion are both injective since $F$ is a global function field.
\end{proof}

\begin{remark}
\begin{romanenum}
\item
The isomorphism $\rho\colon W_F^\ab\to C_F$ depends only on $F$ (and not on our choices of $\infty$, $\varepsilon$, and $\phi\in X_\varepsilon$).
\item
Our proof only uses class field theory to prove that $\rho$ is injective, i.e., to show that we have constructed all finite abelian extensions of $F$.
\end{romanenum}
\end{remark}

\section{Proof of lemmas} \label{S:lemma proofs}

\subsection{Proof of Lemma~\ref{L:lambda basics}} \label{SS:lambda basics} 
\noindent  (i) By Lemma~\ref{L:ideal action}(\ref{I:ideal action, multiplication}) and Lemma~\ref{L:ideal-Galois actions}, we have  $(\sigma\gamma)\phi=(\aA\bB)*\phi$ and $\sigma(\phi_\bB)=\sigma(\phi)_\bB$.  By Lemma~\ref{L:ideal action}(\ref{I:ideal action, multiplication}) and our choice of $\aA$, we find that $\phi_{\aA\bB}=(\aA*\phi)_\bB\cdot \phi_{\aA}=\sigma(\phi)_\bB \phi_\aA=\sigma(\phi_\bB) \phi_\aA$.  We have $\sigma(\phi_\bB(\sigma^{-1}(\xi)))=\sigma(\phi)_\bB(\xi)$ for all $\xi\in \sigma(\phi)[\lambda^e]$, so $V_\lambda(\sigma)\circ V_\lambda(\phi_\bB) \circ V_\lambda(\sigma)^{-1}$ and $V_\lambda(\sigma(\phi_\bB))$ are the same automorphism of $V_\lambda(\sigma(\phi))$.  Therefore,
\begin{align*}
V_\lambda(\phi_{\aA\bB})^{-1}\circ V_\lambda(\sigma\gamma) &= V_\lambda( \phi_\aA)^{-1} \circ V_\lambda(\sigma(\phi_\bB))^{-1} \circ V_\lambda(\sigma\gamma)\\
&= V_\lambda( \phi_\aA)^{-1} \circ (V_\lambda(\sigma)\circ V_\lambda(\phi_\bB) \circ V_\lambda(\sigma)^{-1})^{-1} \circ V_\lambda(\sigma\gamma)\\
&= V_\lambda( \phi_\aA)^{-1} \circ V_\lambda(\sigma)\circ V_\lambda(\phi_\bB)^{-1} \circ V_\lambda(\sigma)^{-1} \circ V_\lambda(\sigma\gamma)\\
&= (V_\lambda( \phi_\aA)^{-1} \circ V_\lambda(\sigma))\circ( V_\lambda(\phi_\bB)^{-1} \circ V_\lambda(\gamma)).
\end{align*}
Thus $\rho^{\aA\bB}_\lambda(\sigma\gamma)=\rho^{\aA}_\lambda(\sigma) \rho^{\bB}_\lambda(\gamma)$.

\noindent (ii)  Since $\aA*\phi = \bB*\phi$, Lemma~\ref{P:homogenous space} implies that the fractional ideal $\bB \aA^{-1}$ is the identity class in $\Pic^+(A)$.   There are thus non-zero $w_1,w_2 \in A$ such that $(w_1)\aA = (w_2) \bB$ and $\varepsilon(w_1)=\varepsilon(w_2) = 1$.    In particular, $w:=w_1/w_2$ is the unique generator of $\bB\aA^{-1}$ satisfying $\varepsilon(w)=1$.  By Lemma~\ref{L:ideal action}(\ref{I:ideal action, principal}), we have $(w_1)*\phi=\phi$ and $\phi_{(w_1)}=\phi_{w_1}$.    Therefore, by Lemma~\ref{L:ideal action}(\ref{I:ideal action, multiplication}) we have
\[
\phi_{\aA(w_1)}=((w_1)*\phi)_{\aA} \cdot \phi_{(w_1)}= \phi_\aA  \phi_{w_1}.
\]
Similarly, $\phi_{\bB(w_2)}= \phi_\bB  \phi_{w_2}$, and hence $\phi_\aA \phi_{w_1} = \phi_\bB \phi_{w_2}$.   Thus
\[
\rho_{\lambda}^\aA(\sigma) \rho_{\lambda}^\bB(\sigma)^{-1} = V_\lambda(\phi_\aA)^{-1} \circ V_\lambda(\phi_\bB) = V_\lambda(\phi_{w_1}) \circ V_\lambda(\phi_{w_2})^{-1}.
\]
The automorphisms $V_\lambda(\phi_{w_1})$ and $V_\lambda(\phi_{w_2})$ both belong to $\Aut_{F_\lambda}(V_\lambda(\phi))=F_\lambda^\times$ and correspond to $w_1$ and $w_2$, respectively.  We conclude that $\rho_{\lambda}^\aA(\sigma) \rho_{\lambda}^\bB(\sigma)^{-1}= w_1 w_2^{-1}=w$.

\noindent (iii)  We view $\lambda$ as both a place of $F$ and a non-zero prime ideal of $A$.  For each $e\geq f$, the kernel of $\phi[\lambda^e] \to \phi[\lambda^e],\, \xi\mapsto \sigma^{-1}(\phi_\aA(\xi))$ has cardinality $|\phi[\lambda^f]|=N(\lambda)^f$.  So $\rho_\lambda^\aA(\sigma)^{-1}=V_\lambda(\sigma^{-1})\circ V_\lambda(\phi_\aA)$, which is an element of $\Aut_{F_\lambda}(V_\lambda(\phi))=F_\lambda^\times$, gives an $\OO_\lambda$-module homomorphism $T_\lambda(\phi)\to T_\lambda(\phi)$ whose cokernel has cardinality $N(\lambda)^f$.   Therefore, $\ord_\lambda(\rho_\lambda^\aA(\sigma)^{-1})=f$.  In particular, we have $\ord_\lambda(\rho_\lambda^\aA(\sigma))=-f$.

\subsection{Proof of Lemma~\ref{L:u facts}} \label{SS:u facts}
Fix a non-constant $y\in A$ that satisfies $\varepsilon(y)=1$ and define $h:=-d_\infty\ord_\infty(y)\geq 1$.   Since $\phi$ is $\varepsilon$-normalized and $\varepsilon(y)=1$, we have $\phi_y = \tau^h+\sum_{j=0}^{h-1} b_j \tau^{j}$ for unique $b_j \in H_A^+$.  We set $u= \sum_{i=0}^\infty a_i \tau^{-i}$ with $a_i\in \Fbar$ to be determined where $a_0\neq 0$.  Expanding out the series $\phi_y u$ and $u\tau^{h}$, we find that $\phi_y u = u\tau^{h}$ holds if and only if 
\begin{equation} \label{E:artin-schreier}
a_i^{q^h} - a_i = - \sum_{\substack{0\leq j\leq h-1,\\ i+j-h\geq 0}}  b_j a_{i+j-h}^{q^j}
\end{equation}
holds for all $i\geq 0$.  We can use the equations (\ref{E:artin-schreier}) to recursively solve for $a_0\neq 0,a_1,a_2,\ldots$.  The $a_i$ belong to $F^\sep$ since (\ref{E:artin-schreier}) is a separable polynomial in $a_i$ and the $b_j$ belong to $H_A^+\subseteq F^\sep$.  Let $k_h$ be the degree $h$ extension of $k$ in $\kbar$.  The elements of the ring $\Fbar\twistedLaurent$ that commute with $\tau^h$ are $k_h\twistedLaurent$.  Since $\tau^h$ belongs to the commutative ring $u^{-1}\phi(F_\infty) u$, we find that $u^{-1}\phi(F_\infty) u$ is a subset of $k_h\twistedLaurent$.   Thus $u \in \Fbar\twistedLaurent^\times$ has coefficients in $F^\sep$ and satisfies $u^{-1} \phi(F_\infty) u \subseteq \kbar\twistedLaurent$.\\

Recall that $\phi$ induces an embedding $\FF_\infty\hookrightarrow L$; this gives inclusions $\FF_\infty\subseteq \kbar \subseteq \Lbar$.  Fix a uniformizer $\pi$ of $F_\infty$.  There is a unique homomorphism $\iota\colon F_\infty \to \kbar\twistedLaurent$ that satisfies the following conditions:
\begin{itemize}
\item
$\iota(x)=x$ for all $x\in \FF_\infty$,
\item $\iota(\pi)=\tau^{-d_\infty}$,
\item $\ord_{\tau^{-1}}(\iota(x))=d_\infty \ord_\infty(x)$ for all $x\in F_\infty^\times$.
\end{itemize}
We have $\iota(F_\infty)=\FF_\infty(\!( \tau^{-d_\infty})\!)$.   Let $C$ be the centralizer of $\iota(F_\infty)$ in $\Fbar\twistedLaurent$.   Using that $\FF_\infty$ and $\tau^{d_\infty}$ are in $\iota(F_\infty)$, we find that $C=\FF_\infty(\!(\tau^{-d_\infty})\!)=\iota(F_\infty)$.

Take any $v\in \Fbar\twistedseries^\times$ that satisfies $v^{-1} \phi(F_\infty) v \subseteq \kbar\twistedLaurent$.  By \cite{MR2018826}*{Lemma~2.3}, there exist $w_1$ and $w_2\in \kbar\twistedseries^\times$ such that
\[
w_1^{-1}(u^{-1} \phi_x u) w_1 = \iota(x) = w_2^{-1}(v^{-1} \phi_x v) w_2
\]
for all $x\in F_\infty$.  So for all $x\in F_\infty$, we have $(uw_1)  \iota(x) (uw_1)^{-1} =\phi_x = (v w_2)  \iota(x) (v w_2)^{-1}$ and hence
\[
(w_2^{-1} v^{-1} u w_1) \iota(x) (w_2^{-1} v^{-1} u w_1)^{-1} = \iota(x).
\]
Thus $w_2^{-1} v^{-1} u w_1$ belongs to $C\subseteq \kbar\twistedLaurent$, and so $v=uw$ for some $w\in \kbar\twistedseries^\times$.

The coefficients of $v$ thus lie in $F^\sep$ since the coefficients of $u$ lie in $F^\sep$ and $w$ has coefficients in the perfect field $\kbar\subseteq F^\sep$.  This completes the proof of (i).  Since $w$ has coefficients in $\kbar$, we have $\sigma(w)=\tau^{\deg(\sigma)} w \tau^{-\deg(\sigma)}$ and hence
\begin{align*}
\sigma(v) \tau^{\deg(\sigma)} v^{-1} &= \sigma(uw)\tau^{\deg(\sigma)} (uw)^{-1}\\
&= \sigma(u) \sigma(w) \tau^{\deg(\sigma)} w^{-1} u^{-1}\\
& = \sigma(u) (\tau^{\deg(\sigma)} w \tau^{-{\deg(\sigma)}})\tau^{\deg(\sigma)} w^{-1} u^{-1}\\
&=\sigma(u) \tau^{\deg(\sigma)} u^{-1}.
\end{align*}
This proves that $\phi_{\aA}^{-1} \sigma(u) \tau^{\deg(\sigma)} u^{-1}$ is independent of the initial choice of $u$.

We now show that $ \phi_{\aA}^{-1} \sigma(u) \tau^{\deg(\sigma)} u^{-1}$ belongs to $\phi(F_\infty)$.   We have seen that $\phi(F_\infty)$ is conjugate to $\iota(F_\infty)$ in $\Fbar\twistedLaurent$ and that $\iota(F_\infty)$ is its own centralizer in $\Fbar\twistedLaurent$.   Therefore, the centralizer of $\phi(F_\infty)$ in $\Fbar\twistedLaurent$ is $\phi(F_\infty)$.  So it suffices to prove that $ \phi_{\aA}^{-1} \sigma(u) \tau^{\deg(\sigma)} u^{-1}$ commutes with $\phi_x$ for all $x\in A$.  Take any $x\in A$.  We have $\sigma(u^{-1} \phi_x u)= \tau^{\deg \sigma} (u^{-1} \phi_x u) \tau^{-\deg \sigma}$ since $u^{-1} \phi_x u$ has coefficients in $\kbar$.   By our choice of $\aA$, we have $\sigma(\phi)_x = (\aA*\phi)_x= \phi_\aA \phi_x \phi_\aA^{-1}$.  Therefore,
\[
\tau^{\deg \sigma} (u^{-1} \phi_x u) \tau^{-\deg \sigma} = \sigma(u^{-1} \phi_x u) = \sigma(u)^{-1} \sigma(\phi)_x \sigma(u) = \sigma(u)^{-1} \phi_\aA \phi_x \phi_\aA^{-1} \sigma(u)
\]
and one concludes that $\phi_\aA^{-1} \sigma(u)\tau^{\deg \sigma} u^{-1}$ commutes with $\phi_x$.

The only thing that remains to be proved is that $\phi_\aA^{-1} \sigma(u)\tau^{\deg \sigma} u^{-1}$ belongs to $\phi(F^+_\infty)$.  Since $\phi$ is $\varepsilon$-normalized, this is equivalent to showing that the first non-zero coefficient of the Laurent series $\phi_\aA^{-1} \sigma(u)\tau^{\deg \sigma} u^{-1}$ in $\tau^{-1}$ is 1.   Since $\phi_\aA$ is a monic polynomial of $\tau$, we need only show that the first non-zero coefficient of $\sigma(u)\tau^{\deg \sigma} u^{-1}\in \Fbar\twistedLaurent$ is 1, i.e., $\sigma(a_0) a_0^{-q^{\deg \sigma}}=1$.      This is true since $a_0\in \kbar^\times$; indeed, $a_0$ is non-zero and satisfies $a_0^{q^h}-a_0=0$ by (\ref{E:artin-schreier}).

\begin{remark} \label{R:integrality}
Take any place $\p\neq \infty$ of $F$ and any valuation $v\colon F^\sep \to \QQ\cup\{+\infty\}$ extending $\ord_\p$.    We have $v(b_j) \geq 0$ for $0\leq j\leq h$ (in \S\ref{SS:Hayes}, we noted that the coefficients of $\phi_y$ are integral over $A$).  Using (\ref{E:artin-schreier}) repeatedly, we find that $v(a_i) \geq 0$ for all $i\geq 0$ (the roots of (\ref{E:artin-schreier}), as a polynomial in $a_i$, differ by a value in $k_h$).

For each $i\geq 1$, the extension $F(a_{i})/F(a_{i-1})$ is an Artin-Schreier extension.  Since the right-hand side of (\ref{E:artin-schreier}) is integral at each place not lying over $\infty$, we deduce that $F(a_{i})/F(a_{i-1})$ is unramified at all places not lying over $\infty$.

Let $L\subseteq F^\sep$ be the extension of $F$ generated by $\kbar$ and the set $\{a_i\}_{i\geq 0}$, i.e., the extension of $F\kbar$ generated by the coefficients of $u$.  We find that $L$ is unramified at all places of $F$ away from $\infty$.  
\end{remark}

\subsection{Proof of Lemma~\ref{L:infty basics}}
Fix a series $u\in \Fbar\twistedseries^\times$ as in Lemma~\ref{L:u facts}(\ref{I:u exists}). 

\noindent (i)   We have $\phi(\rho_\infty^{\aA \bB}(\sigma\gamma))= \phi_{\aA\bB}^{-1} \sigma\gamma(u) \tau^{\deg(\sigma\gamma)} u^{-1}$ and $\phi(\rho_\infty^\bB(\gamma))= \phi_\bB^{-1}\gamma(u)\tau^{\deg(\gamma)} u^{-1}$.  So it suffices to show that $\phi(\rho^\aA_\infty(\sigma))$ equals
\begin{align*}
\phi(\rho_\infty^{\aA \bB}(\sigma\gamma) \rho_\infty^\bB(\gamma)^{-1} )&= \phi_{\aA\bB}^{-1} \,\sigma\gamma(u) \tau^{\deg(\sigma\gamma)} u^{-1} \cdot u \tau^{-\deg(\gamma)}\gamma(u)^{-1} \phi_\bB\\
&= \phi_{\aA\bB}^{-1} \sigma(\gamma(u)) \tau^{\deg(\sigma)}  \gamma(u)^{-1} \phi_\bB\\
&= \phi_{\aA\bB}^{-1}\sigma(\phi_\bB) \phi_\aA \cdot \phi_\aA^{-1} \sigma(\phi_\bB^{-1}\gamma(u)) \tau^{\deg(\sigma)}  (\phi_\bB^{-1} \gamma(u))^{-1}.
\end{align*}
We showed in \S\ref{SS:lambda basics} that $ \phi_{\aA\bB}=\sigma(\phi_\bB) \phi_\aA$, so we need only prove that $\phi_\aA^{-1} \sigma(\phi_\bB^{-1}\gamma(u)) \tau^{\deg(\sigma)}  (\phi_\bB^{-1} \gamma(u))^{-1} $ and $\phi(\rho^\aA_\infty(\sigma))$ agree.  By Lemma~\ref{L:u facts}(\ref{I:u independence}), it thus suffices to show that $(\phi_\bB^{-1}\gamma(u))^{-1}\phi(F_\infty) \phi_\bB^{-1}\gamma(u) \subseteq \kbar\twistedLaurent$.  

Take any $x\in F_\infty$.   Since $u^{-1}\phi_x u$ belongs to $\kbar\twistedLaurent$, so does $\gamma(u^{-1} \phi_x u)$.  By our choice of $\bB$, we have
\[
\gamma(u^{-1} \phi_x u)= \gamma(u)^{-1} \gamma(\phi)_x \gamma(u) = \gamma(u)^{-1} (\bB*\phi)_x \gamma(u) = \gamma(u)^{-1} \phi_\bB \phi_x \phi_\bB^{-1} \gamma(u),
\]
and hence $(\phi_\bB^{-1}\gamma(u))^{-1}  \phi_x (\phi_\bB^{-1} \gamma(u))$ is an element of $\kbar\twistedLaurent$.

\noindent (ii)
First note that
\[
\phi(\rho_\infty^{\aA}(\sigma) \rho_\infty^{\bB}(\sigma)^{-1}) = \phi_\aA^{-1} \sigma(u)\tau^{\deg \sigma} u^{-1} \cdot u \tau^{-\deg \sigma} \sigma(u)^{-1} \phi_\bB=\phi_\aA^{-1} \phi_\bB.
\]
In \S\ref{SS:lambda basics}, we showed that $\phi_\aA^{-1} \phi_\bB = \phi_{w_1}\phi_{w_2}^{-1}$ where $w_1$ and $w_2$ belong to $A$ and $w=w_1/w_2$ is the unique generator of $\bB\aA^{-1}$ satisfying $\varepsilon(w)=1$.   Therefore, $\rho_\infty^{\aA}(\sigma) \rho_\infty^{\bB}(\sigma)^{-1}=w$ as desired.

\subsection{Proof of Lemma~\ref{L:rho-infty is continuous}}
The map $\rho_\infty$ is a homomorphism by Lemma~\ref{L:infty basics}(\ref{I:infty basics-mult}).  Fix a series $u=\sum_{i\geq 0} c_i \tau^{-i}$ as in Lemma~\ref{L:u facts}(\ref{I:u exists}).  For $\sigma \in W_{H_A^+}$, we have $\ord_\infty(\rho_\infty(\sigma))= d_\infty \ord_{\tau^{-1}}(\phi(\rho_\infty(\sigma))) = -d_\infty \deg(\sigma)$.

So to prove that $\rho_\infty$ is continuous, we need only show that $\Gal(F^\sep/H_A^+\kbar) \xrightarrow{\rho_\infty} \OO_\infty^\times$ is continuous.    It suffices to show that for each $e\geq 1$, the homomorphism
\[
\beta_e \colon  \Gal(F^\sep/H_A^+\kbar) \xrightarrow{\rho_\infty} \OO_\infty^\times \to (\OO_\infty/\mathfrak{m}_\infty^e)^\times
\]
has open kernel.  For each $\sigma \in  \Gal(F^\sep/H_A^+\kbar)$, we have $\phi(\rho_\infty(\sigma))= \sigma(u) u^{-1}$.  One can check that $\beta_e(\sigma)=1$, equivalently $\ord_\infty(\rho_\infty(\sigma)-1)\geq e$, if and only if $\ord_{\tau^{-1}}(\sigma(u) u^{-1} - 1)=\ord_{\tau^{-1}}(\sigma(u)-u)$ is at least $ed_\infty$.   Thus the kernel of $\beta_e$ is $\Gal(F^\sep/L_e)$ where $L_e$ is the finite extension of $H_A^+\kbar$ generated by the set $\{c_i\}_{0\leq i< ed_\infty}$

It remains to prove that $\rho_\infty$ is unramified at all places of $H_A^+$ not lying over $\infty$.   Let $L'$ be the subfield of $F^\sep$ fixed by $\ker(\rho_\infty)$; it is the extension of $H_A^+\kbar$ generated by the set $\{c_i\}_{i\geq 0}$.  The field $L'$ does not depend on the choice of $u$, since $\rho_\infty$ does not.   In Remark~\ref{R:integrality}, we saw that the extension $L$ of $F\kbar$ generated by the coefficients of a particular $u$ was unramified at all places of $F$ away from $\infty$.   Therefore, $L'$ is unramified at all places of $F$ away from $\infty$, since $L$ and $H_A^+$ both have this property.

\subsection{Proof of Lemma~\ref{L:Frobenius}} 
First consider the case $\lambda \neq \infty$.  For each $e\geq 1$, we have $\Frob_\p(\xi) = \phi_\p(\xi)$ for all $\xi \in \phi[\lambda^e] \subseteq F^\sep$; this was observed by Hayes \cite{MR1196509}*{p.28}, for a proof see \cite{MR1423131}*{\S7.5}.  Thus the map $V_\lambda(\phi_\p)\colon V_\lambda(\phi) \to V_\lambda(\phi)$ equals $V_\lambda(\Frob_\p)$, and hence $\rho^\p_\lambda(\Frob_\p)=V_\lambda(\phi_\p)^{-1}\circ V_\lambda(\Frob_\p)=1$.

We may now assume that $\lambda=\infty$.  Let $\Fbar_\p$ be an algebraic closure of $F_\p$.   The field $\Fbar_\p$ has a unique place extending the $\p$-adic place of $F_\p$ and let $\bbar\OO_\p\subseteq \Fbar_\p$ be the corresponding valuation ring.   The residue field of $\bbar\OO_\p$ is an algebraic closure of $\FF_\p$ that we denote by $\FFbar_\p$, and let $r_\p\colon \bbar\OO_\p \to \FFbar_\p$ be the reduction homomorphism.  

Choose an embedding $\Fbar \hookrightarrow \Fbar_\p$.   The restriction map $\Gal(\Fbar_\p/F_\p)\to \Gal(\Fbar/F)$ is an inclusion that is well-defined up to conjugation.  So after conjugating, we may assume that $\Frob_\p$ lies in $\Gal(\Fbar_\p/F_\p)$; it thus acts on $\bbar\OO_\p$ and we have $r_\p(\Frob_\p(\xi))= r_\p(\xi)^{N(\p)}$ for all $\xi \in \bbar\OO_\p$.  We will also denote by $r_\p$ the map $\bbar\OO_\p\twistedLaurent \to\FFbar_\p\twistedLaurent$ where one reduces the coefficients of the series.

In \S\ref{SS:Hayes}, we noted that the coefficients of $\phi_x$ are integral over $A$ for each $x\in A$.  This implies that $\phi(A)\subseteq \bbar\OO_\p[\tau]$.  Define the homomorphism
\[
\bbar \phi \colon A \xrightarrow{\phi} \bbar\OO_\p[\tau] \xrightarrow{r_\p} \FFbar_\p[\tau].
\]
The map $\bbar\phi$ is a Drinfeld module over $\FFbar_{\p}$ of rank 1 since $\phi$ is normalized.    Since $\phi$ is normalized, we find that $\phi_x \in \bbar\OO_\p\twistedLaurent^\times$ for all non-zero $x\in A$.   Therefore, $\phi(F) \subseteq \bbar\OO_\p\twistedLaurent$.    Using that $\bbar\OO_\p\twistedLaurent$ is complete with respect to $\ord_{\tau^{-1}}$ and that (\ref{E:rank defn}) holds with $r=1$, we deduce that $\phi(F_\infty) \subseteq \bbar\OO_\p\twistedLaurent$.

We claim that $r_\p$ induces an isomorphism $\phi(F_\infty) \to \bbar\phi(F_\infty)$.   Since $\bbar\phi$ is a Drinfeld module, and hence injective, the map $r_\p$ induces an isomorphism $\phi(A)\to \bbar\phi(A)$ which then extends to an isomorphism $\phi(F)\to \bbar\phi(F)$ of their quotient fields.   The map $r_\p\colon \phi(F)\to \bbar\phi(F)$ preserves the valuation $\ord_{\tau^{-1}}$ so by the uniqueness of completions, $r_\p$ also gives an isomorphism $\phi(F_\infty)\to \bbar\phi(F_\infty)$.

Thus to prove that $\rho_\infty^\p(\Frob_\p)$ equals 1, it suffices to show that $r_\p(\phi(\rho_\infty^\p(\Frob_\p)))=1$.   Lemma~\ref{L:u facts}(\ref{I:u exists}), along with Remark~\ref{R:integrality}, shows that there is a series $u \in \bbar\OO_\p\twistedseries^\times$ with coefficients in $F^\sep$ such that $u^{-1}\phi(F_\infty) u \subseteq \kbar\twistedLaurent$.  With such a series $u$, we have
\[
\phi(\rho_\infty(\Frob_\p)) = \phi_\p^{-1} \Frob_\p(u) \tau^{d} u^{-1}
\]
where $d=\deg(\Frob_\p)$.    The polynomial $\phi_\p \in F[\tau]$ has coefficients in $\OO_\p$ and $r_\p(\phi_\p)=\tau^d$; thus $\phi_\p^{-1}$ belongs to $\bbar\OO_\p\twistedLaurent$ and $r_\p(\phi_\p^{-1})=\tau^{-d}$.  We have $r_\p(\Frob_\p(u))=\tau^d r_\p(u) \tau^{-d}$.   Therefore,
\[
r_\p(\phi(\rho_\infty(\Frob_\p))) = r_\p(\phi_\p^{-1}) r_\p(\Frob_\p(u)) \tau^{d} r_\p(u)^{-1}= \tau^{-d} \cdot \tau^d r_\p(u) \tau^{-d} \cdot \tau^d r_\p(u)^{-1}=1.
\]

\section{The rational function field} \label{S:k(t)}

We return to the rational function field $F=k(t)$ where $k$ is a finite field with $q$ elements.   Using our constructions, we shall recover the description of $F^\ab$ given by Hayes in \cite{MR0330106}; he expressed $F^\ab$ as the compositum of three linearly disjoint fields over $F$.  In particular, we will explain how two of these fields arise naturally from our representation $\rho_\infty$.

We define $A=k[t]$; it is the subring of $F$ consisting of functions that are regular away from a unique place $\infty$ of $F$.  We have $\FF_\infty=k$, $F_\infty=k(\!(t^{-1})\!)$ and $\ord_\infty\colon F_\infty^\times \to \ZZ$ is the valuation for which $\ord_\infty(t^{-1})=1$. Let $\varepsilon\colon F_\infty^\times \to k^\times$ be the unique sign function of $F_\infty$ that satisfies $\varepsilon(t^{-1})=1$.   Those $x\in F_\infty^\times$ for which $\varepsilon(x)=1$ form the subgroup $F_\infty^+:= \ang{t}(1+t^{-1} k[\![t^{-1}]\!])=\ang{t}(1+\mathfrak{m}_\infty)$.\\

Recall that the Carlitz module is the homomorphism $\phi\colon A \to F[\tau],\, a\mapsto \phi_a$ of $k$-algebras for which $\phi_t = t + \tau$.   In the notation of \S\ref{SS:Hayes}, $\phi$ is a Hayes module  for $\varepsilon$.   The coefficients of $\phi_t$ lie in $F$, so the normalizing field $H_A^+$ for $(F,\infty,\varepsilon)$ equals $F$.   We saw that $\Gal(H_A^+/F)$ acts transitively on $X_\varepsilon$, so $X_\varepsilon =\{\phi\}$.

For each place $\lambda$ of $F$ (including $\infty$!), we have defined a continuous homomorphism $\rho_\lambda\colon W_F^\ab \to F_\lambda^\times,\, \sigma \mapsto \rho^A_\lambda(\sigma)$.    The representation $\rho_\lambda$ is characterized by the property that 
\[
\rho_\lambda(\Frob_\p)=\p
\] 
holds for each monic irreducible polynomial $\p\in A=k[t]$ not corresponding to $\lambda$ (combine Lemmas~\ref{L:lambda basics}(\ref{I:lambda basics-defined}), \ref{L:infty basics}(\ref{I:infty basics-defined}) and \ref{L:Frobenius} to show that $\rho^A_\lambda(\Frob_\p) \rho_\lambda^{(\p)}(\Frob_\p)^{-1} = \rho^A_\lambda(\Frob_\p)$ is the unique generator $w$ of $\p A$ that satisfies $\varepsilon(w)=1$).      For $\lambda\neq \infty$, the representation $\rho_\lambda$ has image in $\OO_\lambda^\times$, so it extends to a continuous representation $\Gal(F^\ab/F)\to \OO_\lambda^\times$.   The  image of $\rho_\infty$ lies in $F_\infty^+$, but is unbounded (so it does not extend to a Galois representation). 

 Combining the representations $\rho_\lambda$ together, we obtain a continuous homomorphism
\begin{equation} \label{E:Carlitz expression}
\prod_\lambda \rho_\lambda \colon W_F^{\ab} \to F_\infty^+\times \prod_{\lambda\neq \infty} \OO_\lambda^\times.
\end{equation}
By composing $\prod_\lambda \rho_\lambda$ with the quotient map $\AA_F^\times \to C_F$, we obtain a continuous homomorphism $\rho\colon W_F^\ab \to C_F$.    Corollary~\ref{C:main} says that $\rho$ is an isomorphism of topological groups (and that the inverse of $\sigma \mapsto \rho(\sigma)^{-1}$ is the Artin map $\theta_F$).    The quotient map $ F_\infty^+\times \prod_{\lambda\neq \infty} \OO_\lambda^\times\to C_F$ is actually an isomorphism, so the map (\ref{E:Carlitz expression}) is also an isomorphism.  Taking profinite completions, we obtain an isomorphism
\begin{equation} \label{E:Carlitz expression 2}
\Gal(F^\ab/F) \xrightarrow{\sim} \widehat{F_\infty^+}\times \prod_{\lambda\neq \infty} \OO_\lambda^\times= \widehat{\ang{t}} \cdot (1+\mathfrak{m}_\infty) \times \prod_{\lambda\neq \infty} \OO_\lambda^\times.
\end{equation}
Using this isomorphism, we can now describe three linearly disjoint abelian extensions of $F$ whose compositum is $F^\ab$.

\subsubsection{Torsion points}
The representation $\chi:=\prod_{\lambda\neq \infty}\rho \colon \Gal(F^\ab/F)\to \prod_{\lambda\neq\infty} \OO_\lambda^\times = \widehat{A}^\times$ arises from the Galois action on the torsion points of $\phi$ as described in \S\ref{SS:rational field}.  The fixed field in $F^\ab$ of $\ker(\chi)$ is the field $K_\infty:= \cup_{m} F(\phi[m])$ where the union is over all monic polynomials $m$ of $A$, and we have an isomorphism $\Gal(K_\infty/F)\xrightarrow{\sim} \widehat{A}^\times$.    The field $K_\infty$, which was first given be Carlitz, is a geometric extension of $F$ that is tamely ramified at $\infty$.

\subsubsection{Extension of constants}
Define the homomorphism $\deg \colon W_F^\ab \to \ZZ$ by $\sigma\mapsto -\ord_\infty(\rho_\infty(\sigma))$; this agrees with our usual definition of $\deg(\sigma)$ [it is easy to show that $\ord_{\tau^{-1}}(\phi(\rho_\infty(\sigma)))$ equals $\ord_{\tau^{-1}}(\tau^{\deg(\sigma)})=-\deg(\sigma)$, and then use (\ref{E:rank defn}) with $r=d_\infty=1$].  The map $\deg$ thus factors through $W(\kbar(t)/k(t))\xrightarrow{\sim} \ZZ$, where $W(\kbar(t)/k(t))$ is the group of $\sigma\in \Gal(\kbar(t)/k(t))$ that act on $\kbar$ as an integral power of $q$.  Of course, $\kbar(t)/k(t)$ is an abelian extension with $\Gal(\kbar(t)/k(t))\xrightarrow{\sim} \widehat{\ZZ}$.

\subsubsection{Wildly ramified extension}
Define the homomorphism 
\[
W_F^\ab \to 1+ \mathfrak{m}_\infty,\quad \sigma\mapsto \rho_\infty(\sigma)/t^{\deg(\sigma)};
\]
it is well-defined since $\ord_\infty(\rho_\infty(\sigma))=-\deg(\sigma)$ and since the image of $\rho_\infty$ is contained in $F^+_\infty=\ang{t} (1+\mathfrak{m}_\infty)$.    Since $1+\mathfrak{m}_\infty$ is compact, this gives rise to a Galois representation
\[
\beta\colon \Gal(F^\ab/F) \to 1+ \mathfrak{m}_\infty.
\]
Let $L_\infty$ be the fixed field of $\ker(\beta)$ in $F^\ab$.   The field  $L_\infty/F$ is an abelian extension of $F$ that is unramified away from $\infty$ and is wildly ramified at $\infty$; it is also a geometric extension of $F$.

We will now give an explicit description of this field.  Define $a_0=1$ and for $i\geq 1$, we recursively choose $a_i \in F^\sep$ that satisfy the equation
\begin{equation} \label{E:artin-schreier Carlitz}
a_i^{q} - a_i = - t a_{i-1}.
\end{equation}
This gives rise to a chain of field extensions, $F\subset F(a_1) \subset F(a_2) \subset F(a_3) \subset \cdots$.     We claim that $L_\infty=\bigcup _i F(a_i)$. 

The construction of $\rho_\infty$ starts by finding an appropriate series $u\in \Fbar\twistedseries^\times$ with coefficients in $F^\sep$.    Define $u:=\sum_{i\geq 0} a_i \tau^{-i}$ with the $a_i$ defined as above.   The recursive equations that the $a_i$ satisfy give us that $\phi_t u = u\tau$ (just multiply them out and check!).  As shown in \S\ref{SS:u facts}, this implies that $u^{-1} \phi(F_\infty) u \subseteq \kbar\twistedLaurent$.    For $\sigma\in W_F^\ab$, $\rho_\infty(\sigma)$ is then the unique element of $F_\infty^+$ for which $\phi(\rho_\infty(\sigma))= \sigma(u)\tau^{\deg(\sigma)} u^{-1}$ where $\sigma(u)$ is obtained by letting $\sigma$ act on the coefficients of $u$.   In particular, $\phi(\beta(\sigma))=\sigma(u)u^{-1}$ for all $\sigma \in \Gal(L_\infty/F)$ since $L_\infty/F$ is a geometric extension.   We find that $\beta(\sigma)=1$ if and only if $\sigma(u)=u$, and thus $L_\infty$ is the extension of $F$ generated by the set $\{a_i\}_{i\geq 1}$.\\

Using the isomorphism (\ref{E:Carlitz expression 2}), we find that $F^\ab$ is the compositum of the fields $K_\infty$, $\kbar(t)$ and $L_\infty$, and that they are linearly disjoint over $F$.  This is exactly the description given by Hayes in \cite{MR0330106}*{\S5}; the advantage of our representation $\rho_\infty$ is that the fields $\kbar(t)$ and $L_\infty$ arise naturally.


\end{document}